\documentclass[leqno]{article}
\usepackage[T1]{fontenc}
\usepackage[utf8]{inputenc}
\usepackage[english]{babel}
\usepackage{geometry}
\usepackage{graphicx}
\usepackage{mathtools}        
\usepackage{amsfonts,amssymb,mathrsfs,amscd,amsthm}
\usepackage{verbatim}
\usepackage{url}
\usepackage{stmaryrd}
\usepackage{hyperref}

\theoremstyle{plain}

\newtheorem{thm}{Theorem}[section]
\newtheorem{prop}[thm]{Proposition}
\newtheorem{cor}[thm]{Corollary}

\newtheorem{lem}[thm]{Lemma}

\theoremstyle{definition}
\newtheorem{conj}[thm]{Conjecture}
\newtheorem{examp}[thm]{Example}

\newtheorem{dfn}[thm]{Definition}
\newtheorem{rk}[thm]{Remark}

\newtheorem{notation}[thm]{Notation}

\renewcommand{\a}{\alpha}
\renewcommand{\b}{\beta}

\newcommand{\D}{\Delta}
\newcommand{\g}{\gamma}
\newcommand{\e}{\varepsilon}
\renewcommand{\k}{\varkappa}

\newcommand{\cR}{\mathcal{R}}
\newcommand{\cM}{\mathcal{M}}
\newcommand{\cT}{\mathcal{T}}
\newcommand{\cS}{\mathcal{S}}
\newcommand{\cGH}{\mathcal{GH}}
\newcommand{\cC}{\mathcal{C}}

\newcommand{\N}{\mathbb{N}}

\newcommand{\R}{\mathbb{R}}
\newcommand{\Z}{\mathbb{Z}}

\newcommand{\diam}{{\operatorname{diam}}}
\newcommand{\dis}{{\operatorname{dis}}}
\newcommand{\codis}{{\operatorname{codis}}}
\newcommand{\dil}{{\operatorname{dil}}}

\renewcommand{\:}{\colon}

\title{The Gromov--Hausdorff distance between $l^p$-products of metric spaces}
\author{Emin Abdullaev}

\begin{document}

\maketitle

\begin{abstract}
This paper studies $l^p$-products of metric spaces and provides estimates for the Gromov--Hausdorff distances between them. The case of linear products is considered separately, and sufficient conditions for attainability of the estimates are given for it. Examples of calculating the Gromov--Hausdorff distance between flat tori are given. It is proved that for any metric space $X$ of density $d(X)$, the Gromov--Hausdorff distance between it and its $l^\infty$-product (in which the number of factors corresponds to $d(X)$) is equal to half its diameter.
\end{abstract}

\tableofcontents
\newpage

\section*{Introduction}

The Gromov--Hausdorff space $\cM$, consisting of all nonempty metric compact sets considered up to isometry, is one of the central objects of modern metric geometry.

The foundations of this approach were laid by F. Hausdorff in his work \cite{Hausdorff}, who proposed a metric for closed bounded subsets of a metric space. Later, D. Edwards \cite{Edwards} and M. Gromov \cite{Gromov} extended this idea to the set of all metric spaces via isometric embeddings, leading to the creation of the Gromov--Hausdorff distance.

Although this function can take infinite values in the general case, it is a metric on the set $\cM$. It has been proven that the resulting Gromov--Hausdorff space is path-connected, complete, separable, and geodesic \cite{Gromov, BurBurIva, IvaNikTuz}.

However, directly computing the Gromov--Hausdorff distance
is an $NP$-hard \cite{NP01, NP02} problem, as it requires minimizing distortion over all possible correspondences between spaces. Therefore, finding two-sided estimates for specific classes of metric spaces is of fundamental importance.

The issues of computing and estimating the Gromov--Hausdorff distance for various classes of metric spaces have been studied in detail in many papers. For example, G. Adams \cite{Adams} and F. Memoli \cite{LimMemoly} obtained estimates for spheres, and in the works \cite{IvaTuzSimplex, GrigIvaTuz} a technique was developed that made it possible to calculate distances for a wide class of objects, including simplices and ultrametric spaces. In the work \cite{IvaTuzBorsuk}, using the Borsuk number, sufficient conditions were given for the attainability of the upper bound (that is, when the Gromov--Hausdorff distance between two metric spaces is equal to half the maximum of their diameters). Other examples include calculations of the Gromov--Hausdorff distance between a segment and a circle \cite{JiTuzhilin} and between $\Z^2$ and $\R^2$ with the Euclidean metric \cite{MikhAsymptDim}.

The concept of the continuous Gromov--Hausdorff distance received its modern development in the work \cite{LimMemoly}.
An important stage in the development of the theory was the work of S. A. Bogaty and A. A. Tuzhilin \cite{BogatyTuzhilin} on the study of the properties of the continuous Gromov--Hausdorff distance.
An alternative version of the continuous Gromov--Hausdorff distance is also considered in the work \cite{LeeMorales} in the context of comparing dynamical systems. In this work, the authors use a construction that does not possess the triangle inequality property, which determines the specificity of their approach.

This paper is devoted to obtaining estimates for the Gromov--Hausdorff distance between $l^p$-products of metric spaces. In Section 2, two-sided estimates for $l^p$-products in the general case are obtained, and some of these results are extended to the continuous Gromov--Hausdorff distance. In Section 3, the upper bound is improved for linear $l^p$-products and sufficient conditions are given for this bound to be sharp. In particular, examples are given for $l^1$-products and for flat tori, where a formula for computing the distance is obtained.
Section 4 considers the Gromov--Hausdorff distance between a metric space and its $l^{\infty}$-product with itself and proves that
for any metric space $X$ of density $d(X)$, the Gromov--Hausdorff distance between it and its $l^\infty$-product (in which the number of factors corresponds to $d(X)$) is equal to half its diameter. In particular, if $X$ is separable, then the Gromov--Hausdorff distance between it and its countable $l^\infty$-product is equal to half its diameter. Section 5 provides an example of applying these results to the Gromov-Hausdorff topological distance.

\section*{Acknowledgments}

The author expresses gratitude to his supervisor, Doctor of Physical and Mathematical Sciences, Professor A.A. Tuzhilin, and Doctor of Physical and Mathematical Sciences, Professor A.O. Ivanov, for posing the problem and their continued attention to the work.

\section{Basic Definitions and Preliminary Results}

The Gromov--Hausdorff distance is a standard metric on the set of isometric classes of nonempty compact metric spaces.

Let $\cGH$ be the class of all nonempty metric spaces, and $\cM \subset \cGH$ be the class of all compact metric spaces.
For arbitrary $X, Y \in \cGH$, let $d_X$ and $d_Y$ denote their metrics, and let $\diam(X) = \sup_{x, y \in X} d_X(x, y)$ be the diameter of the space. A \emph{correspondence} between $X$ and $Y$ is a subset $R \subset X \times Y$ whose projections onto $X$ and $Y$ are surjective. The set of all correspondences between $X$ and $Y$ will be denoted by $\mathcal{R}(X, Y)$.

The \emph{distortion} of a correspondence $R \in \mathcal{R}(X, Y)$ is the quantity
$$ \dis R = \sup \bigl\{ |d_X(x, x') - d_Y(y, y')| : (x, y), (x', y') \in R \bigr\}. $$
Then the Gromov--Hausdorff distance is defined as
$$ d_{GH}(X, Y) = \frac{1}{2} \inf \{ \dis R : R \in \mathcal{R}(X, Y) \}. $$

Next, following ~\cite{LimMemoly}, \cite{BogatyTuzhilin}, we define the continuous Gromov--Hausdorff distance $d_{GH}^c$. For a pair of continuous mappings $f\: X \to Y$ and $g\: Y \to X$, we define their distortions and \emph{codistortion}:
\begin{itemize}
\item $\dis(f) = \sup_{x, x' \in X} |d_X(x, x') - d_Y(f(x), f(x'))|$;
\item $\dis(g) = \sup_{y, y' \in Y} |d_Y(y, y') - d_X(g(y), g(y'))|$;
\item $\codis(f, g) = \sup_{x \in X, y \in Y} |d_X(x, g(y)) - d_Y(f(x), y)|$.
\end{itemize}
The distance $d_{GH}^c(X, Y)$ is defined as
$$ d_{GH}^c(X, Y) = \frac{1}{2} \inf_{f, g} \max \{ \dis f, \dis g, \codis(f, g) \}, $$
where the infimum is taken over all continuous mappings $f \: X \to Y$ and $g \: Y \to X$.

In what follows, when discussing the properties of the Gromov--Hausdorff distance, we follow \cite{BurBurIva} and \cite{BogatyTuzhilin}.

For an arbitrary metric space $(X, d_X)$, we define the \emph{diameter}
$$\diam(X) = \sup_{x, \tilde{x} \in X} d(x, \tilde{x}).$$

For any bounded metric spaces $X, Y$, the following estimate holds:
$$ \frac{1}{2} |\diam(X) - \diam(Y)| \le d_{GH}(X, Y) \le d_{GH}^{c}(X, Y) \le \frac{1}{2} \max\{\diam(X), \diam(Y)\}. $$

Denote $\{pt\}$ as a metric space consisting of a single point. Then
for any metric space $X$, the following equality holds:
$$d_{GH}(X, \{pt\}) = d_{GH}^{c}(X, \{pt\}) = \frac{1}{2} \diam(X).$$

For any $X, Y \in \cGH$ and arbitrary mappings $f \: X \to Y$ and $g \: Y \to X$, we denote $R_{f, g} = \{ (x, f(x)) \ | \ x \in X)\} \cup \{ (g(y), y) \ | \ y \in Y \}$, then $R_{f, g}\in \cR(X, Y)$.
\begin{prop}[\cite{BogatyTuzhilin}]
For any $X, Y \in \cGH$ and arbitrary mappings $f \: X \to Y$ and $g \: Y \to X$, we have
$$\dis(R_{f, g}) = \max\{\dis(f), \dis(g), \codis(f, g)\}.$$
\end{prop}

\begin{dfn} Let $\{(X_{n}, d_n)\}_{n=1}^{\infty}$ be a sequence of metric spaces such that $\sum_{n=1}^{\infty}(\diam(X_n))^{p} < \infty$, where $1 \le p < \infty$. Then, on the Cartesian product $\prod_{n=1}^{\infty} X_{n}$, we introduce the metric
$$ d\bigl(\{x_n\}_{n=1}^{\infty}, \{y_n\}_{n=1}^{\infty} \bigr) = \Biggl(\sum_{n = 1}^{\infty} \bigl(d_n(x_{n}, y_{n}) \bigr)^{p} \Biggr)^{1/p}$$ and the resulting metric space $(l^{p})\prod_{n} X_{n}$ will be called the $l^p$-\emph{product} of the spaces $X_n$.
\end{dfn}

\begin{rk}\label{rk:iso}
Let $\{(X_{n, m}, d_{X_{n, m}})\}_{n, m \in \N}$ be a family of metric spaces, such that \hfill\break
$\sum_{n, m = 1}^{\infty} \diam(X_{n, m})^p < \infty$, where $1 \le p < \infty$.
Then the following spaces are isometric:
$$(l^p)\prod_{n=1}^{\infty} \Biggl( (l^p )\prod_{m=1}^{\infty} (X_{n, m}, d_{X_{n, m}})\Biggr) = (l^p)\prod_{n, m=1}^{\infty} (X_{n, m}, d_{X_{n, m}}).$$
\end{rk}

\begin{dfn}
Let $\{(X_{\alpha}, d_{\a})\}_{\alpha\in A}$ be an arbitrary family of metric spaces such that \hfill\break
$\sup_{\alpha \in A} \diam(X_{\alpha})<\infty$.
Then, on the Cartesian product $\prod_{\a \in A} X_{\a}$, we introduce the metric
$$ d\bigl(\{x_{\alpha}\}_{\alpha \in A}, \{y_{\alpha}\}_{\alpha \in A} \bigr) = \sup_{\alpha \in A} d_{\a}(x_{\alpha}, y_{\alpha})$$ and call the resulting metric space $(l^{\infty})\prod_{\alpha \in A} X_{\alpha}$ the $l^{\infty}$-\emph{product} of the spaces $X_{\a}$.
\end{dfn}

\begin{notation}
For an arbitrary set $A$, denote its cardinality by $\#A$.
\end{notation}

\begin{notation}
For an arbitrary cardinal number $\kappa$, let $\D_{\kappa}$ denote the metric space of cardinality $\#\D_{\kappa}=\kappa$ for which all nonzero distances are equal to $1$.
\end{notation}

\section{General Estimates for Distances Between Products}
We begin with a simple observation.

\begin{lem}\label{lem:matan}
Let $\{A_n \}_{n \in \N} \subset [0, +\infty)$ be a sequence of subsets of the real line. Set $A = \sum_{n=1}^{\infty} A_n := \bigl\{ \sum_{n=1}^{\infty} a_n \ | \ a_n \in A_n \bigr\}$.
Then $\sup A = \sum_{n = 1}^{\infty} \sup A_n$ and $\inf A = \sum_{n = 1}^{\infty} \inf A_n$.
\end{lem}

The following lemma estimates the Gromov--Hausdorff distance between products of metric spaces in terms of the distance between their components.

\begin{lem}\label{lem:zero}
Let $\{(X_{n}, d_n)\}_{n=1}^{\infty}$ and $\{(Y_{n}, \varrho_n)\}_{n=1}^{\infty}$ be sequences of metric spaces such that \hfill\break
$\sum_{n = 1}^{\infty} \bigl(\diam(X_n) \bigr)^{p} < \infty$ and
$\sum_{n = 1}^{\infty} \bigl(\diam(Y_n) \bigr)^{p} < \infty$, where $1 \le p < \infty$.
For any $R_n\in \cR(X_n,Y_n)$, we can think of the Cartesian product $R=\prod_{n}R_n$ as a correspondence in \hfill\break $\cR\bigl((l^p)\prod_n X_n, (l^p) \prod_n Y_n \bigr)$. Then
$$\dis(R) \le \biggl(\sum_{n}\bigl(\dis(R_n)\bigr)^p\biggr)^{1/p}.$$
\end{lem}

\begin{proof}
Let us put $(X, d_X) = (l^{p})\prod_{n} (X_{n}, d_n)$ and $(Y, d_Y) = (l^{p})\prod_{n} (Y_{n}, d_n)$.
Then we get that
\begin{multline*} \dis(R) =
\dis \biggl(\prod_{n} R_{n} \biggr) =
\sup_{(x, y), (\tilde{x}, \tilde{y}) \in R}
\bigl| d_X(x, \tilde{x}) - d_Y(y, \tilde{y}) \bigr| = \\
\
= \sup_{(x, y), (\tilde{x}, \tilde{y}) \in R}
\Bigl| \|{\{d_n(x_{n}, \tilde{x}_n)\}_{n=1}^{\infty}}\|_{l^p} - \|{\{\varrho_n(y_{n}, \tilde{y}_n)\}_{n=1}^{\infty}} \|_{l^p} \Bigr| \le \\
\
\le
\sup_{(x, y), (\tilde{x}, \tilde{y}) \in R}
 \| \{d_n(x_{n}, \tilde{x}_{n}) - \varrho_n(y_{n}, \tilde{y}_{n}) \}_{n=1}^{\infty} \|_{l^{p}},
\end{multline*}
where the last inequality is the inequality
$\bigl|\|a\|-\|b\|\bigr| \le \|a-b\|$ for the $l^p$-norm and vectors $a_n = d_n(x_n, \tilde{x}_n), \ b_n = \varrho_n(y_n, \tilde{y}_n)$.

Let's put
$$A_n = \Bigl\{ \bigl| d_n(x_{n}, \tilde{x}_{n}) - \varrho_n(y_{n}, \tilde{y}_{n}) \bigr|^{p} \ \Big|
\ (x_n, y_n), (\tilde{x}_n, \tilde{y}_n) \in R_n \Bigr\},$$
$$A = \biggl\{ \sum_{n=1}^{\infty} \bigl| d_n(x_{n}, \tilde{x}_{n}) - \varrho_n(y_{n}, \tilde{y}_{n}) \bigr|^{p} \ \Big|
\ (x, y), (\tilde{x}, \tilde{y}) \in R \biggr\}.$$
Applying the lemma ~\ref{lem:matan}, we find that
\begin{multline*}
\dis(R)\le
\sup_{(x, y), (\tilde{x}, \tilde{y}) \in R}
\Biggl( \sum_{n=1}^{\infty} \left| d_n(x_{n}, \tilde{x}_{n}) - \varrho_n(y_{n}, \tilde{y}_{n}) \right|^{p} \Biggr)^{1/p} = \\
\
= \Biggl(
\sum_{n=1}^{\infty} \biggl(
\sup_{(x_n, y_n), (\tilde{x}_n, \tilde{y}_n) \in R_n}
\bigl| d_n(x_{n}, \tilde{x}_{n}) - \varrho_n(y_{n}, \tilde{y}_{n}) \bigr|
\biggr)^{p} \Biggr)^{1/p}
= \Biggl( \sum_{n=1}^{\infty} \bigl( \dis(R_{n}) \bigr)^{p} \Biggr)^{1/p}.
\end{multline*}
\end{proof}

\begin{lem}\label{lem:EstimationInGeneralCase01}
Let $\{(X_{n}, d_n)\}_{n=1}^{\infty}$ and $\{(Y_{n}, \varrho_n)\}_{n=1}^{\infty}$ be sequences of metric spaces such that \hfill\break
$\sum_{n = 1}^{\infty} \bigl(\diam(X_n) \bigr)^{p} < \infty$ and $\sum_{n = 1}^{\infty} \bigl(\diam(Y_n) \bigr)^{p} < \infty$, where $1 \le p < \infty$.
Then
$$d_{GH}\biggl((l^{p})\prod_{n} X_{n}, \ (l^{p})\prod_{n} Y_{n} \biggr) \le \Biggl( \sum_{n} (d_{GH} \bigl(X_{n}, Y_{n})\bigr)^{p} \Biggr)^{1/p}.$$
\end{lem}

\begin{proof}
Let us put $(X, d_X) = (l^{p})\prod_{n} (X_{n}, d_n)$ and $(Y, d_Y) = (l^{p})\prod_{n} (Y_{n}, d_n)$.
For any $\e > 0$, consider $R_n \in \cR(X_n, Y_n)$ such that
$$ \bigl\| \{\dis(R_n)\}_{n=1}^{\infty} \bigr\|_{l^p} \le
2 \bigl\| \{d_{GH}(X_n, Y_n)\}_{n=1}^{\infty} \bigr\|_{l^p} + \e.$$
Set $R = \prod_{n=1}^{\infty} R_n \in \cR(X, Y)$.
Then, by the definition of the Gromov--Hausdorff distance and Lemma ~\ref{lem:zero}, we obtain
$$2d_{GH}(X, Y) \le \dis(R) \le \bigl\| \{\dis(R_n)\}_{n=1}^{\infty} \bigr\|_{l^p} \le
2\bigl\| \{d_{GH}(X_n, Y_n)\}_{n=1}^{\infty} \bigr\|_{l^p} + \e.$$
Passing to the limit at $\e \to 0$, we obtain what we need.
\end{proof}

\begin{examp}
Consider two rectangles on the Euclidean plane: $X = (l^2)[0, A] \times [0, B]$ and $Y = (l^{2})[0, C] \times [0, D]$.
Then $2d_{GH}([0, A], [0, C])=|A-C|$ and $2d_{GH}([0, B], [0, D])=|B-D|$.
From the previous lemma, we obtain that $2d_{GH}(X, Y) \le \sqrt{(A-C)^2+(B-D)^2}$.
\end{examp}

Now we obtain the inverse estimate.

\begin{lem}\label{lem:EstimationInGeneralCase02}
Let $\{X_{n}\}_{n=1}^{\infty}$ and $\{Y_{n}\}_{n=1}^{\infty}$ be sequences of metric spaces such that \hfill\break
$\sum_{n = 1}^{\infty} \bigl(\diam(X_n) \bigr)^{p} < \infty$ and
$\sum_{n = 1}^{\infty} \bigl(\diam(Y_n) \bigr)^{p} < \infty$, where $1 \le p < \infty$.
Then
\begin{multline*}
\sup_{n} \Biggl(
d_{GH}(X_{n}, Y_{n}) -
\frac{1}{2}\biggl( \sum_{m \ne n} \bigl( \diam(X_{m})\bigr)^{p} \biggr)^{1/p}
-\frac{1}{2}\biggl( \sum_{m \ne n} \bigl( \diam(Y_{m})\bigr)^{p} \biggr)^{1/p}
\Biggr) \le \\
\le d_{GH}\biggl((l^{p})\prod_{n} X_{n}, \ (l^{p})\prod_{n} Y_{n} \biggr).
\end{multline*}
\end{lem}

\begin{proof}
Set $X = (l^{p})\prod_{n} X_{n}$ and $Y = (l^{p})\prod_{n} Y_{n}$.
Note that any correspondence $R \in \cR(X, Y)$ can be projected onto $X_{n} \times Y_{n}$ to obtain a correspondence $P^{R}_{n} \in \cR(X_{n}, Y_{n})$. In other words, a pair $(\hat{x}_{n}, \hat{y}_{n})$ is contained in $P^{R}_{n}$ if and only if there exists a pair $(x,y) \in R$ such that $x_n = \hat{x}_{n}$ and $\ y_{n} = \hat{y}_{n}$.
Further, we will denote the metric in the spaces $X_n$, $Y_{n}$, $X$, and $Y$ by $d_{X_n}$, $d_{Y_n}$, $d_{X}$, and $d_{Y}$, respectively.
Then
\begin{multline*}
d_{GH}(X_{n}, Y_n) = \frac{1}{2} \inf_{R \in \cR(X_{n}, Y_{n})} \dis(R) = \frac{1}{2} \inf_{R \in \cR(X, Y)} \dis(P_{n}^{R}) =  \\
\
= \frac{1}{2} \inf_{R \in \cR(X, Y)} \sup_{\substack{(x, y) \in R \\ (\tilde{x}, \tilde{y}) \in R}}
\bigl| d_{X_{n}}(x_{n}, \tilde{x}_{n}) - d_{Y_{n}}(y_{n}, \tilde{y}_{n}) \bigr| = \\
\
= \frac{1}{2} \inf_{R \in \cR(X, Y)} \sup_{\substack{(x, y) \in R \\ (\tilde{x}, \tilde{y}) \in R}}
\Bigl| \bigl(d_{X}(x, \tilde{x}) - d_{Y}(y, \tilde{y}) \bigr) +
\bigl(d_{X_{n}}(x_{n}, \tilde{x}_{n}) - d_{X}(x, \tilde{x}) \bigr) +
\bigl(d_{Y}(y, \tilde{y}) - d_{Y_{n}}(y_{n}, \tilde{y}_{n}) \bigr) \Bigr|.
\end{multline*}
First, we evaluate the resulting expression using the triangle inequality. Then we write out the inequality $\bigl|\|a\|_{l^p}-\|b\|_{l^p}\bigr|\le \|a-b\|_{l^p}$, written first for $a_m=d_{Y_m}(x_m, \tilde{x}_m)$ and $b_m=d_n(x_n, \tilde{x}_n)$ for $m=n$ and $b_m=0$ for $m\ne n$, and then for
$a_m=d_{Y_m}(y_m, \tilde{y}_m)$ and $b_m=d_n(y_n, \tilde{y}_n)$ for $m=n$ and $b_m=0$ for $m\ne n$. Thus, we obtain
\begin{multline*}
d_{GH}(X_n, Y_n) \le \frac{1}{2} \inf_{R \in \cR(X, Y)} \sup_{\substack{(x, y) \in R \\ (\tilde{x}, \tilde{y}) \in R}}
\bigl| d_{X}(x, \tilde{x}) - d_{Y}(y, \tilde{y}) \bigr|
+ \\ + \frac{1}{2} \sup_{\substack{x, \tilde{x} \in X \\ y, \tilde{y} \in Y}}
\Biggl(
\biggl| \Bigl( \sum_{m = 1}^{\infty} \bigl( d_{X_{m}}(x_{m}, \tilde{x}_{m}) \bigr) ^{p} \Bigr)^{1/p} - d_{X_{n}}(x_{n}, \tilde{x}_{n}) \biggr| + \\
+ \biggl| \Bigl( \sum_{m = 1}^{\infty} \bigl( d_{Y_{m}}(y_{m}, \tilde{y}_{m}) \bigr) ^{p} \Bigr)^{1/p} - d_{Y_{n}}(y_{n}, \tilde{y}_{n}) \biggr|
\Biggr) \le \\
\
\le d_{GH}(X, Y) +
\frac{1}{2}
\sup_{\substack{x, \tilde{x} \in X \\ y, \tilde{y} \in Y}}
\Biggl( \sum_{m \ne n} \bigl( d_{X_m}(x_m, \tilde{x}_m) \bigr)^{p} \Biggr)^{1/p} +
\sup_{\substack{x, \tilde{x} \in X \\ y, \tilde{y} \in Y}}
\frac{1}{2} \Biggl( \sum_{m \ne n} \bigl( d_{Y_m}(y_m, \tilde{y}_m)) \bigr)^{p} \Biggr)^{1/p} \le \\
\
\le d_{GH}(X, Y) +
\frac{1}{2} \Biggl( \sum_{m \ne n} \bigl( \diam(X_{m}) \bigr)^{p} \Biggr)^{1/p} +
\frac{1}{2} \Biggl( \sum_{m \ne n} \bigl( \diam(Y_{m}) \bigr)^{p} \Biggr)^{1/p}.
\end{multline*}
Thus, for each $n$, we get
$$d_{GH}(X_{n}, Y_{n}) -
\frac{1}{2}\biggl( \sum_{m \ne n} \bigl( \diam(X_{m})\bigr)^{p} \biggr)^{1/p}
-\frac{1}{2}\biggl( \sum_{m \ne n} \bigl( \diam(Y_{m})\bigr)^{p} \biggr)^{1/p} \le
d_{GH}(X, Y), $$
which proves the lemma.
\end{proof}

\begin{examp}
Let $X, Y$ be metric spaces of the same diameter $\diam(X) = \diam(Y) = \alpha < 1/2$, and let $\kappa \ne \kappa'$ be two different cardinal numbers. We denote $\D=\D_{\kappa}$ and $\D'=\D_{\kappa'}$. It is well known that $2d_{GH}(\Delta_{}, \Delta')=1$. We set $X_1 = \Delta, \ X_2 = X, \ Y_1 = \Delta', \ Y_2 = Y$. Then, by Lemma ~\ref{lem:EstimationInGeneralCase02} we obtain
\begin{multline*}
2d_{GH} \bigl((l^{p}) \ X \times \Delta, (l^{p}) \ Y \times \Delta' \bigr) =
2d_{GH}((l^p)X_1 \times X_2, (l^p) Y_1 \times Y_2) \ge \\
\ge 2d_{GH}(X_1, Y_1) - \diam(X_2) - \diam(Y_2) > 1-2\alpha.
\end{multline*}
Note that in this case, the simple estimate
$0=\bigl|\diam((l^{p}) \ X \times \Delta) - \diam((l^{p}) \ Y \times \Delta') \bigr| \le 2d_{GH} \bigl((l^{p}) \ X \times \Delta, (l^{p}) \ Y \times \Delta' \bigr)$ does not imply anything (that is, this estimate is trivial).
\end{examp}

The two lemmas above yield the following theorem.
\begin{thm}\label{thm:EstimationInGeneralCase}
Let $\{X_{n}\}_{n=1}^{\infty}$ and $\{Y_{n}\}_{n=1}^{\infty}$ be sequences of metric spaces, such that \hfill\break
$\sum_{n = 1}^{\infty} \bigl(\diam(X_n) \bigr)^{p} < \infty$ and $\sum_{n = 1}^{\infty} \bigl(\diam(Y_n) \bigr)^{p} < \infty$, where $1 \le p < \infty$.
Then
\begin{multline*}
\sup_{n} \Biggl(
d_{GH}(X_{n}, Y_{n}) -
\frac{1}{2}\biggl( \sum_{m \ne n} \bigl( \diam(X_{m})\bigr)^{p} \biggr)^{1/p}
-\frac{1}{2}\biggl( \sum_{m \ne n} \bigl( \diam(Y_{m})\bigr)^{p} \biggr)^{1/p}
\Biggr) \le \\
\le d_{GH}\biggl((l^{p})\prod_{n} X_{n}, \ (l^{p})\prod_{n} Y_{n} \biggr)
\le \Biggl( \sum_{n=1}^{\infty}(d_{GH} \bigl(X_{n}, Y_{n})\bigr)^{p} \Biggr)^{1/p}.
\end{multline*}
\end{thm}

The following theorem is proved similarly.
\begin{thm}\label{thm:EstimationInGeneralCase_L_infty}
Let $\{X_{\a} \}_{\a \in A}$ and $\{Y_{\a} \}_{\a \in A}$ be arbitrary families of metric spaces, such that \hfill\break
$\sup_{\a \in A} \diam(X_{\a})<\infty$ and $\sup_{\a \in A} \diam(Y_{\a})<\infty$.
Then
\begin{multline*}
\sup_{\a \in A} \Bigl( d_{GH}(X_{\a}, Y_{\a})-\frac{1}{2} \sup_{\b \ne \a} \diam(X_{\b})-\frac{1}{2} \sup_{\b \ne \a} \diam(Y_{\b}) \Bigr) \le \\ \le d_{GH} \Biggl((l^{\infty}) \prod_{\a \in A} X_{\a}, \ (l^{\infty})\prod_{\a \in A}Y_{\a} \Biggr) \le \sup_{\a \in A}d_{GH}(X_{\a}, {Y_{a}}).\end{multline*}
\end{thm}

\begin{rk}
Note that in Theorem ~\ref{thm:EstimationInGeneralCase_L_infty}, the upper bound holds for arbitrary (even unbounded) metric spaces, as is clear from the proof.
\end{rk}

\begin{dfn}[\cite{BurBurIva}]
Let $X, Y$ be metric spaces. A mapping $f \: X \to Y$ is called \emph{Lipschitz} if there exists $C \ge 0$ such that $d_Y(f(x_1), f(x_2)) \le C d_X(x_1, x_2)$ for all $x_1, x_2 \in X$.
Any suitable value of $C$ is called a constant Lipschitz mapping of $f$. The minimum Lipschitz constant is called the \emph{extension} of $f$ and is denoted by $\dil(f)$.
\end{dfn}

\begin{rk}
\begin{enumerate}
\item
Consider the mapping $\cT_p$, which associates to each compact metric space $X$ its square, $\cT_p \: \cM \to \cM, \ X \mapsto (l^p)X \times X$.
Then the mapping's dilation $\cT_p = 2^{1/p}$ is Lipschitz continuous and, in particular, continuous.
Moreover, the dilation $\dil(\cT_p) = 2^{1/p}$.
\item
The mapping $$\cT_{\infty} \: X\mapsto (l^{\infty})\prod_{\a \in A}X$$ is 1-Lipschitz continuous and, in particular, continuous. Moreover, the dilation $\dil(\cT_{\infty}) = 1$. \end{enumerate}
\end{rk}

\begin{proof}
Let us justify the first point (the second is obtained similarly). Since, according to \ref{lem:EstimationInGeneralCase01}, $d_{GH}((l^p)X \times X, (l^p)Y \times Y) \le 2^{1/p} d_{GH}(X, Y)$, the mapping $\cT_p$
is Lipschitz continuous with constant $C \ge 2^{1/p}$. It remains to provide compact metric spaces $X, Y$ such that $d_{GH}((l^p)X \times X, (l^p)Y \times Y) = 2^{1/p} d_{GH}(X, Y)$. This will certainly be true if, for example, $X = \{pt\}$, and $Y$ is an arbitrary metric compact space.
\end{proof}

\subsection{Remarks on the theorem: the case of continuous Gromov--Hausdorff distance}

We extend some of the obtained results to the case of continuous Gromov--Hausdorff distance $d_{GH}^{c}$.
The following proposition in this case is indeed verbatim for $d_{GH}^{c}$.
\begin{prop}\label{prop:SimpleLinf}
Let $(X_1, d_{X_1}) \dots, (X_N, d_{X_N})$ and $(Y_1, d_{Y_1}) \dots, (Y_N, d_{Y_N})$ be finite sequences of metric spaces. Let $1 \le p < \infty$, denote $X = (l^p) X_1 \times \dots \times X_N, \ Y = (l^p) Y_1 \times \dots \times Y_N$.
Then
$$d_{GH}^{c}(X, Y) \le \Biggl( \sum_{n} (d_{GH}^{c} \bigl(X_{n}, Y_{n})\bigr)^{p} \Biggr)^{1/p}.$$
\end{prop}

\begin{proof}
We will proceed similarly to the proof of Lemma ~\ref{lem:EstimationInGeneralCase01}.
For any $\e > 0$, consider continuous mappings $f_n \: X_n \to Y_n, \g_n \: Y_n \to X_n$ such that
$$\bigl\| \{\max\{\dis(f_n), \dis(g_n), \codis(f_n, g_n) \}\}_{n=1}^{N} \bigr\|_{l^p} \le
2\bigl\| \{d_{GH}^{c}(X_n, Y_n)\}_{n=1}^{N} \bigr\|_{l^p} + \e.$$
Let's put
$R = R_{f_1, g_1} \times \dots \times R_{f_N, g_N} \in \cR(X, Y)$. Let us define continuous mappings $f \: X \to Y, \ (x_1, \dots, x_N) \mapsto (f_1(x_1), \dots, f_n(x_n))$ and $g \: Y \to X, \ (y_1, \dots, y_n) \mapsto (g_1(y_1), \dots g_n(y_n))$.
Then, by the definition of the Gromov--Hausdorff distance and by the lemma ~\ref{lem:zero} we obtain
\begin{multline*}
2d_{GH}^{c}(X, Y) \le
\max\{ \dis(f), \dis(g), \codis(f, g)\} \le \dis(R) \le \bigl\| \{\dis(R_{f_n,g_n})\}_{n=1}^{N} \bigr\|_{l^p} = \\ = \bigl\| \{ \max\{\dis(f_n), \dis(g_n), \codis(f_n, g_n)\} \}_{n=1}^{N} \bigr\|_{l^p}
\le 2 \bigl\| \{d_{GH}^{c}(X_n, Y_n)\}_{n=1}^{N} \bigr\|_{l^p} + \e.
\end{multline*}
Passing to the limit at $\e \to 0$, we obtain what we need.
\end{proof}

Now let's generalize this to arbitrary sequences.

\begin{lem}\label{lem:ProductContinuity}
Let $\{(X_{n}, d_n)\}_{n=1}^{\infty}$ be a sequence of metric spaces such that $\sum_{n = 1}^{\infty} \bigl(\diam(X_n) \bigr)^{p} < \infty$, where $1 \le p < \infty$.
Then
$$d_{GH}^{c}\biggl((l^{p})\prod_{n=1}^{\infty} X_{n}, \ (l^{p})\prod_{n=1}^{N} X_{n} \biggr) \le
\frac{1}{2}\biggl(\sum_{n=N+1}^{\infty}\bigl(\diam(X_n)\bigr)^p\biggr)^{1/p}.$$
In particular, $$\lim_{N \to \infty} d_{GH}^{c}\biggl((l^{p})\prod_{n=1}^{\infty} X_{n}, \ (l^{p})\prod_{n=1}^{N} X_{n} \biggr) = 0.$$
\end{lem}

\begin{proof}
Let's put $\tilde{X}_1 = \tilde{Y}_1 = (l^p)\prod_{n=1}^{N} X_n$,
$\tilde{X}_2 = (l^p) \prod_{n=N+1}^{\infty} X_n$,
$\tilde{Y}_2 = \{pt\}$.
Since $d_{GH}^{c} (\tilde{X}_1, \tilde{Y}_1)=0$, then by Proposition
~\ref{prop:SimpleLinf} and Remark ~\ref{rk:iso} we have
\begin{multline*}
d_{GH}^{c}\biggl((l^{p})\prod_{n=1}^{\infty} X_{n}, \ (l^{p})\prod_{n=1}^{N} X_{n} \biggr) =
d_{GH}^{c} \bigl((l^p)\tilde{X}_1 \times \tilde{X}_2, (l^p) \tilde{Y}_1 \times \tilde{Y}_2\bigr) \le d_{GH}^{c}(\tilde{X}_2, \tilde{Y}_2) = \\
= d_{GH}^{c}\biggl((l^p)\prod_{n=N+1}^{\infty}X_n, \{pt\} \biggr) =
\frac{1}{2} \diam \biggl((l^p)\prod_{n=N+1}^{\infty}X_n \biggr) = \frac{1}{2}\biggl(\sum_{n=N+1}^{\infty}\bigl(\diam(X_n)\bigr)^p\biggr)^{1/p}.
\end{multline*}
\end{proof}

\begin{thm}
Let $\{(X_{n}, d_n)\}_{n=1}^{\infty}$ and $\{(Y_{n}, \varrho_n)\}_{n=1}^{\infty}$ be sequences of metric spaces, such that $\sum_{n = 1}^{\infty} \bigl(\diam(X_n) \bigr)^{p} < \infty$ and $\sum_{n = 1}^{\infty} \bigl(\diam(Y_n) \bigr)^{p} < \infty$, where $1 \le p < \infty$.
Then
$$d_{GH}^{c}\biggl((l^{p})\prod_{n=1}^{\infty} X_{n}, \ (l^{p})\prod_{n=1}^{\infty} Y_{n} \biggr) \le \Biggl( \sum_{n=1}^{\infty} (d_{GH}^{c} \bigl(X_{n}, Y_{n})\bigr)^{p} \Biggr)^{1/p}.$$
\end{thm}

\begin{proof}
Since this has already been proved for finite products in Lemma ~\ref{lem:ProductContinuity}, by the triangle inequality we obtain for $N \to \infty$, that
\begin{multline*}
d_{GH}^{c}\biggl((l^{p})\prod_{n=1}^{\infty} X_n, \ (l^{p})\prod_{n=1}^{\infty} Y_{n} \biggr) \le
d_{GH}^{c}\biggl((l^{p})\prod_{n=1}^{N} X_n, \ (l^{p})\prod_{n=1}^{N} Y_{n} \biggr) + \\
d_{GH}^{c}\biggl((l^{p})\prod_{n=1}^{\infty} X_{n}, \ (l^{p})\prod_{n=1}^{N} X_{n} \biggr) +
d_{GH}^{c}\biggl((l^{p})\prod_{n=1}^{\infty} Y_{n}, \ (l^{p})\prod_{n=1}^{N} Y_{n} \biggr) \le \\
\
\le
\Biggl( \sum_{n=1}^{\infty} (d_{GH}^{c} \bigl(X_{n}, Y_{n})\bigr)^{p} \Biggr)^{1/p} +
\frac{1}{2} \Biggl( \sum_{n=N+1}^{\infty} \bigl(\diam(X_{n})\bigr)^{p} \Biggr)^{1/p} +
\frac{1}{2} \Biggl( \sum_{n=N+1}^{\infty} \bigl(\diam(Y_{n})\bigr)^{p} \Biggr)^{1/p},
\end{multline*}
where the last two terms tend to zero.
\end{proof}

\begin{rk}
For $l^{\infty}$-products, the situation is somewhat more complicated. Clearly, for finite sequences of metric spaces
$$(X_1, d_{X_1}), \dots, (X_N, d_{X_N}); \ (Y_1, d_{Y_1}), \dots, (Y_N, d_{Y_N})$$ will hold
$$d_{GH}^{c}((l^{\infty})X_1 \times \dots \times X_N, (l^{\infty})Y_1 \times \dots \times Y_N) \le \max\{ d_{GH}^{c} (X_{1}, Y_{1}), \dots d_{GH}^{c} (X_{N}, Y_{N}) \}, $$
However, in the case of an infinite $l^{\infty}$-product, the proof of the previous proposition cannot be carried over verbatim (the problem arises that the last two terms will stop tending to zero).
If we add the additional condition that $\diam(X_n) \to 0$ and $\diam(Y_n) \to 0$, then the estimate
$$d_{GH}^{c} \biggl( \prod_{n=1}^{\infty} X_n, \ \prod_{n=1}^{\infty} Y_n \biggr) \le \sup_{n \in \N} d_{GH}^{c} (X_n, Y_n).$$
\end{rk}

\begin{conj}
Suppose that for arbitrary sequences of metric spaces
$\{(X_n, d_{X_n}\}_{n \in \N}$, \hfill\break $\{(Y_n, d_{Y_n}\}_{n \in \N}$,
$\sup_{n \in \N} \diam(X_n) < \infty$ and $\sup_{n \in \N} \diam(Y_n) < \infty$.
Then is the estimate
$$d_{GH}^{c} \biggl( \prod_{n=1}^{\infty} X_n, \ \prod_{n=1}^{\infty} Y_n \biggr) \le \sup_{n \in \N} d_{GH}^{c} (X_n, Y_n)$$ of $l^{\infty}$-products correct in the general case?
\end{conj}

\section{Distance Estimation for Linear Products}

In this section, we consider a more special case and improve the estimate obtained in Lemma ~\ref{lem:EstimationInGeneralCase01} for it.

\begin{subsection}{Several Lemmas}

\begin{rk}\label{rk:density}
Let $f \: K \times A \to \R$ and $g \: K \times A \to \R$ be given continuous functions, where $K, A$ are arbitrary nonempty metric compact sets.
We consider them as families of continuous mappings $\{f_a\}_{a\in A}$ and $\{g_a\}_{a\in A}$ defined on the space of continuous functions $\cC(K)$, where the $\sup$-norm is defined on $\cC(K)$.
We also consider continuous mappings $F, G \: \cC(K) \to \R$.
Suppose that on an everywhere dense set of parameters $a \in B \subset A$ the equality
$F(f_a) = G(g_a)$ holds. Then this equality will hold for all $a \in A$.
\end{rk}

\begin{proof}
Indeed, the function $f$ is continuous on the compact set $K \times A$, hence it is uniformly continuous.
Since, by the uniform continuity of $f$ and $g$, the mappings $a \mapsto F(f_a)$ and $a \mapsto G(g_a)$ are continuous and, by assumption, coincide on the dense set $B$, the desired equality will hold for all $a \in A$.
\end{proof}

\begin{lem}\label{lem:TechnicalLemma}
Consider a function $\xi \: [0, +\infty) \to \R$ defined as follows: $$\xi(x) = \bigl| (x+1)^{p} - (\alpha x + \beta)^{p} \bigr|,$$ where $\alpha, \beta>0$ are positive parameters and $0 < p \le 1$.
Then, for every $T>0$, $$\sup_{x \in [0,T]} \xi(x) = \max\{ \xi(0), \xi(T) \}.$$
\end{lem}

\begin{proof}
Without loss of generality, we can assume that $\beta \le 1$. Indeed, if $\beta \ge 1$, then we have
$$\xi(x) = \beta^{p} \biggl| \Bigl(\frac{x}{\beta} + \frac{1}{\beta} \Bigr)^{p} - \Bigl(\frac{\alpha x}{\beta}+1 \Bigr)^{p} \biggr|$$ and, making the substitution $t = \alpha x/\beta$, we obtain
$$\xi(t)=\beta^{p} \biggl| (t+1)^{p} - \Bigl( \frac{t}{\alpha} + \frac{1}{\beta} \Bigr)^{p} \biggr|,$$ where now $1/\beta \le 1$.
We will also assume that $p < 1$, since the case $p = 1$ is trivial.

We denote $\eta(x) = (x+1)^{p} - (\alpha x + \beta)^{p}$, then $\xi(x) = |\eta(x)|$. Note that the derivative
$$\eta'(x_0) = p(x_{0}+1)^{p - 1}-p\alpha(\alpha x_{0}+\beta)^{p - 1} = 0 \iff
x_0 = \frac{\beta - \alpha^{\frac{1}{1-p}}}{\alpha^{\frac{1}{1-p}} - \alpha}.$$
By remark ~\ref{rk:density}, it suffices to prove the lemma on an everywhere dense set of parameters
$\bigl{(\alpha, \beta, p, T)}$, so we can assume that all denominators do not vanish during the calculations.
Now we will consider two cases:

\begin{enumerate}
\item Let the point $x_{0}$ either not exist or $x_{0} \le 0$.
Since we assumed that $\beta \le 1$, this case corresponds to either $\alpha \ge 1$ or $\beta \ge \alpha^{\frac{1}{1-p}}$.
\item Now let the point $x_{0}$ exist and $x_{0} > 0$. Then it turns out that $$0 < \alpha < 1, \ 0 < \beta < \alpha^{\frac{1}{1-p}}.$$
Consequently, $$\left( \frac{\beta - \alpha}{ \alpha^{\frac{1}{1-p}} - \alpha } \right)^{p-2} > 0.$$

Note that the function $\eta(x)$ does not vanish on $[0, +\infty)$. Indeed,
$$\eta(x)=0 \iff x=\frac{\beta-1}{1-\alpha} < 0.$$

Since $\eta(0) = 1 - \beta^{p} > 0$, then
for all $x \in [0, +\infty)$ we have $\xi(x)= \eta(x)>0.$

To complete the proof, it suffices to verify that $\eta''(x_{0})>0$.
So,
$$\eta''(x) = p(p-1)(x+1)^{p-2} - \alpha^2 p(p-1)(\alpha x + \beta)^{p-2}$$
And
$$\eta''(x_0) = p(p-1) \left( \frac{\beta - \alpha}{ \alpha^{\frac{1}{1-p}} - \alpha } \right)^{p-2} \left( 1 - \alpha^{\frac{1}{1-p}} \right) < 0,$$
that is, there will be a local minimum at point $x_{0}$.
\end{enumerate}
\end{proof}

\begin{cor}\label{cor:TechnicalCor}
Consider the function $\xi \: [0, +\infty) \to \R$ defined as follows: $$\xi(x) = \bigl| (x+1)^{1/p} - (\alpha x + \beta)^{1/p} \bigr|,$$ where $\alpha, \beta>0$ are positive parameters and $1 \le p < \infty$.
Then, for every $T>0$, $$\sup_{x \in [0,T]} \xi(x) = \max\{ \xi(0), \xi(T) \}.$$
\end{cor}
\begin{proof}
This follows immediately from ~\ref{lem:TechnicalLemma} by replacing $p$ with $1/p$.
\end{proof}

\begin{lem}\label{lem:BestTechnical}
Consider the function $\xi \: [0, T_1] \times \dots\times [0, T_N] \to \R$, defined as follows:
$$\xi(\k) = \xi(\k_1, \dots, \k_N)= \Biggl|
\biggl( A +\sum_{n=1}^{N} |a_{n} \k_{n}|^{p} \biggr)^{1/p} -
\biggl( B +\sum_{n=1}^{N} |b_{n} \k_{n}|^{p} \biggr)^{1/p}
\Biggr|,$$ where $a_1, \dots, a_N \ge 0$, $b_1, \dots, b_N \ge 0$ is an arbitrary set, $A, B \ge 0$, and $1 \le p < \infty$.
Then
\begin{multline*}
\sup_{\k \in [0, T_1] \times \dots \times [0, T_N]} \xi(\k) = \max_{\k \in \{0, T_1\} \times \dots \times \{0, T_N\}} \xi(\k) = \\
= \max_{S \subset \{1, \dots, N\}}
\Biggl| \biggl(A + \sum_{n \in S}|T_n a_n|^{p} \biggr)^{1/p} - \biggl(B + \sum_{n \in S}|T_n b_n|^{p} \biggr)^{1/p} \Biggr|.
\end{multline*}
\end{lem}

\begin{proof}
It suffices to prove only the first equality, since the second equality is obvious.
As in Lemma ~\ref{lem:TechnicalLemma}, according to Remark ~\ref{rk:density}, we will assume without loss of generality that
$a_1, \dots, a_N > 0$, $b_1, \dots, b_N > 0$, $A, B>0$.

Let $N = 1$, then
$$\xi(\k)=\bigl|(a^p \k^p+A)^{1/p} - (b^p \k^p+B)^{1/p} \bigr|=A^{1/p} \biggl| \Bigl(\frac{a^{p} \k^p}{A} + 1\Bigr)^{1/p} -
\Bigl(\frac{b^{p} \k^p}{A} + \frac{B}{A}\Bigr)^{1/p} \Bigr|,$$ and, making the replacement
$$x = \frac{a^p \k^p}{A}, \ \alpha=\frac{b^p}{a^p}, \ \beta=\frac{B}{A},$$ we reduce this case to the corollary ~\ref{cor:TechnicalCor}.

Now let $N > 1$.
We prove by induction on $N$, assuming that for $N - 1$ it is true.
Thus, applying for each fixed
$\k_{N}=\k_N^* \in [0, \ T_N]$
the induction hypothesis for the function
$\tilde{\xi}(\k_1, \dots, \k_{N-1}) := \xi(\k_1, \dots , \k_{N-1}, \k_N^*)$, we obtain that
$$\sup_{\k \in [0, \ T_1] \times \dots \times [0, \ T_{N-1}]} \xi(\k_1, \dots, \k_{N-1},\k_{N}^*) = \max_{(\k_1, \dots \k_{N-1}) \in \{0, \ T_1\} \times \dots \times\{0, \ T_{N-1}\}} \xi(\k_1, \dots, \k_{N-1},\k_{N}^*).$$
Applying the indoction hypothesis for $N = 1$ and for the function
$\xi^*(\k_{N})=\xi(\k_1^*, \dots \k_{N-1}^*, \k_N)$ for fixed $\k_1^*, \dots \k_{N-1}^*$, we obtain that
$$\sup_{\k_{N}\in[0, T_N]} \xi^*(\k_1^*, \dots, \k_{N-1}^*, \k_N) = \max\{\xi(\k_1^*, \dots, \k_{N-1}^*, 0), \ \xi(\k_1^*, \dots, \k_{N-1}^*, T_{N})\}.$$
Therefore,
\begin{multline*}
\sup_{\k \in [0,  T_1] \times \dots \times [0,  T_N]} \xi(\k_1, \dots,\k_{N}) =
\sup_{\k_{N}^* \in [0,  T_{N}]} \
\sup_{\k \in [0,  T_1] \times \dots \times [0,  T_{N-1}]}
\xi(\k_1, \dots,\k_{N-1}, \k_N^*) = \\
\
= \sup_{\k_{N}^* \in [0, T_{N}]} \
\max_{\k \in \{ 0, T_1 \} \times \dots \times \{0, T_{N-1} \} }
\xi(\k_1, \dots,\k_{N-1}, \k_N^*) = \\
\
= \max_{\k^* \in \{ 0, T_1 \} \times \dots \times \{0, T_{N-1} \} } \
\sup_{\k_{N} \in [0, T_{N}]}
\xi(\k_1^*, \dots,\k_{N-1}^*, \k_N) = \\
\
= \max_{\k^* \in \{ 0, T_1 \} \times \dots \times \{0, T_{N-1} \} } \max\{ \xi(\k_1^*, \dots, \k_{N-1}^*, 0), \ \xi(\k_1^*, \dots, \k_{N-1}^*, T_N) \} = \\
\
= \max_{\k \in \{0, T_1\} \times \dots \times \{0, T_N\}} \xi(\k).
\end{multline*}
\end{proof}

\begin{prop}\label{prop:TheMainTechnical}
\begin{enumerate}
\item Let $a=\{a_n\}_{n \in \N}\in l^p$, $b=\{a_n\}_{n \in \N}\in l^p$ --- two sequences, $1 \le p < \infty$.
We define a function $\xi_p \: [0, 1]^{\N} \to \R$, defined as follows:
$$\xi_p (\k)=\Biggl| \biggl(\sum_{n=1}^{\infty} |a_n \k_n|^p \biggl)^{1/p} -
\biggl(\sum_{n=1}^{\infty} |b_n \k_n|^p \biggl)^{1/p}
\Biggr|. $$
Then $$\sup_{\k \in [0, 1]^{\N}} \xi_p (\k)=\sup_{\k \in \{0, 1\}^{\N}} \xi_p (\k)=\sup_{S \subset \N} \ \Biggl|
\biggl(\sum_{n \in S} |a_n|^p \biggr)^{1/p}- \biggl(\sum_{n \in S} |b_n|^p \biggr)^{1/p} \Biggr|.$$

\item Let $a = \{a_{\alpha}\}_{\alpha \in A}$, $b = \{b_{\alpha}\}_{\alpha \in A}$ be a family of numbers such that $\sup_{\alpha \in A} a_{\alpha} < \infty$ and $\sup_{\alpha \in A} b_{\alpha} < \infty$.
We define a function $\xi_{\infty} :\ [0, 1]^A \to \R$ defined as follows:
$$\xi_{\infty}(\k) = \bigg| \sup_{\alpha \in A}|a_{\alpha}\k_{\alpha}| - \sup_{\alpha \in A}|b_{\alpha}\k_{\alpha}|\bigg|.$$
Then
$$\sup_{\k \in [0, 1]^A} \xi_{\infty}(\k)=\sup_{\k \in \{0, 1\}^A} \xi_{\infty}(\k)= \sup_{S \subset A}\bigg| \sup_{\alpha \in S}|a_{\alpha}| - \sup_{\alpha \in S}|b_{\alpha}|\bigg|.$$
\end{enumerate}
\end{prop}

\begin{proof}
\begin{enumerate}
\item Denote $$\xi_p^{(N)} (\k)=\Biggl| \biggl(\sum_{n=1}^{N} |a_n \k_n|^p \biggl)^{1/p} -
\biggl(\sum_{n=1}^{N} |b_n \k_n|^p \biggl)^{1/p}
\Biggr|.$$ Let $\k = \{\k_{n}\}_{n \in \N}$ be a fixed element. Then, for any $\e > 0$, there exists $N \in \N$ such that $|\xi_p(\k)-\xi_p^{(N)}(\k)|<\e$.
By Lemma ~\ref{lem:BestTechnical}, we obtain that $$\xi_p(\k) \le \e + \xi_p^{(N)}(\k) \le \e + \sup_{\k \in [0, 1]^N} \xi_p^{(N)}(\k) = \e + \sup_{\k \in \{0, 1\}^N} \xi_p^{(N)}(\k) \le \e + \sup_{\k \in \{0, 1\}^{\N}} \xi_p (\k).$$
Letting $\e \to 0$, we obtain the desired result.

\item Let $\k = \{\k_{\a}^*\}_{\a \in A}$ be a fixed element. Then, for any given $\e >0$, there exist $\a, \b$ such that
$$\xi_{\infty}(\k^*) \le \bigl| |a_{\a}||\k_{a}^*|-|b_{\b}||\k_{\b}^*| \bigr| + \e.$$ Let us consider two cases.
\begin{enumerate}
\item Let $\a \ne \b$. Without loss of generality, we assume that $|a_{\a}||\k_{\a}^*|\ge|b_{\b}||\k_{\b}^*|$. Then consider $\k'=\{\k'_{\g}\}_{\g \in A}\in\{0,1\}^A$ such that $\k'_{\g}=1$ for $\g=\a$ and $\k'_{\g}=0$ for $\g \ne \a$. Then
\begin{multline*}
\xi_{\infty}(\k^*) \le \bigl| |a_{\a}||\k_{a}^*|-|b_{\b}||\k_{\b}^*| \bigr| + \e \le \\ \le
\bigl| |a_{\a}||\k'_{a}|-|b_{\b}||\k'_{\b}| \bigr| + \e \le \xi_{\infty}(\k') + \e \le \sup_{\k \in \{0.1\}^A } \xi_{\infty}(\k) + \e. \end{multline*}

\item Let $\a = \b$. Consider $\k'=\{\k'_{\g}\}_{\g \in A} \in \{0,1\}^A$ such that $\k'_{\g}=1$ for $\g=\a$ and $\k'_{\g}=0$ for $\g \ne \a$. Then
$$ \xi_{\infty}(\k^*) \le |\k_{\a}^*|\bigl| |a_{\a}|-|b_{\a}|\bigr|+\e \le |\k'_{\a}|\bigl| |a_{\a}|-|b_{\a}|\bigr|+\e
\le \xi_{\infty}(\k')+\e\le\sup_{\k \in \{0,1\}^A } \xi_{\infty}(\k) + \e.$$
\end{enumerate}
Passing to the limit $\e \to 0$, we obtain $$\xi_{\infty}(\k^*) \le\sup_{\k \in \{0,1\}^A } \xi_{\infty}(\k).$$
\end{enumerate}
\end{proof}

\end{subsection}

\begin{subsection}{Main Theorems}

\begin{notation}
Recall that for an arbitrary metric space $(X, d)$ and for any $\lambda > 0$, we denote by $\lambda (X, d)$ the metric space $(X, \lambda d)$, where the metric is defined by the equality $(\lambda d)(x,y) := \lambda d(x, y)$. For brevity, we will henceforth denote $X$ instead of $(X, d)$ and $\lambda X$ instead of $\lambda (X, d)$.
\end{notation}

\begin{notation}
For any $1 \le p \le\infty$, we set
$$l^{p}_{+}= \bigl\{ \{ x_n \}_{n \in \N} \in l^p \big| \ (\forall n\in\N) \ (x_n \ge 0)\bigr\}.$$
\end{notation}

\begin{dfn}
Let ${(X_n, d_{X_n})}_{n \in \N}$ be a sequence of metric spaces of diameter $\diam(X_n) = 1$.
Let $1 \le p < \infty$.
For any $a = \{a_n\}_{n \in \N} \in l^{p}_{+}$ we define \emph{linear product}
$(X, d_X) = (l^p) \prod_{n=1}^{\infty} a_n (X_n, d_{X_n})$.
\end{dfn}

\begin{prop}\label{prop:DisDiagonal}
\begin{enumerate}
\item \label{prop:DisDiagonal:1} Let $1\le p < \infty$. Consider $\{W_n\}_{n \in \N}$ --- a fixed sequence of metric spaces, let $\diam(W_n)=1$ for all $n\in \N$. Let $a,b \in l^p_{+}$.

Denote $$X = (l^p)\prod_{n=1}^{\infty}a_n W_n, \ Y=(l^p) \prod_{n=1}^{\infty}b_n W_n.$$
Consider the mapping $f \: X \to Y$ defined by the rule $f \: \{w_n\}_{n \in \N} \mapsto \{w_n\}_{n \in \N}$.

Then $$\dis(f)=\sup_{S \subset \N} \ \Biggl| \biggl(\sum_{n \in S} |a_n|^p \biggr)^{1/p}- \biggl(\sum_{n \in S} |b_n|^p \biggr)^{1/p} \Biggr|.$$

\item \label{prop:DisDiagonal:2} Consider $\{W_{\a}\}_{\a \in A}$ --- any family
of metric spaces for which $\diam(W_{\a})=1$ for all $\a \in A$.
Let $a = \{a_{\a}\}_{\a \in A}, a_{\a}\ge0$ and $b = \{b_{\a}\}_{\a \in A}, b_{\a}\ge0$ be such that $\sup_{\a \in A}a_{\a} < \infty$, $\sup_{\a \in A}b_{\a} < \infty$.

Let us denote $$X = (l^{\infty})\prod_{\b \in A} a_{\b}W_{\b}, \ Y = (l^{\infty})\prod_{\b \in A} b_{\b}W_{\b}.$$

Consider the mapping $f \: X \to Y$ given by the rule
$f \: \{w_{\a}\}_{\a \in A} \mapsto \{w_{\a}\}_{\a \in A}$.

Then $$dis(f)= \sup_{S \subset A}\big| \sup_{\b \in S}a_{\b} - \sup_{\b \in S}b_{\b}\big|.$$
\end{enumerate}
\end{prop}

\begin{proof}
We will prove only point ~\ref{prop:DisDiagonal:1}; point ~\ref{prop:DisDiagonal:2} is obtained similarly.
Using the definition of mapping distortion and proposition ~\ref{prop:TheMainTechnical} and denoting $\varkappa_n = d_{W_n}(w_n, \tilde{w}_n)$, we obtain
\begin{multline*}
\dis(f) = \sup_{x, \tilde{x} \in X} \bigl| d_{X}(x, \tilde{x})-d_{Y}(f(x),f(\tilde{x})) \bigr| =\\
\
= \sup_{\substack{w = \{w_n\}_{n = 1}^{\infty} \in X \\ \tilde{w} = \{\tilde{w}_n\}_{n = 1}^{\infty} \in X}} \Biggl| \biggl(\sum_{n=1}^{\infty} \bigl(a_n \ d_{W_n}(w_n, \tilde{w}_n) \bigr)^p \biggl)^{1/p}-
\biggl(\sum_{n=1}^{\infty} \bigl(b_n \ d_{W_n}(w_n, \tilde{w}_n) \bigr)^p \biggl)^{1/p} \Biggr| \le\\
\
\le \sup_{\k \in [0,1]^{\N}} \ \Biggl| \biggl(\sum_{n=1}^{\infty} |a_n \k_n|^p \biggl)^{1/p} -
\biggl(\sum_{n=1}^{\infty} |b_n \k_n|^p \biggl)^{1/p} \Biggr| = \sup_{S \subset \N} \ \Biggl| \biggl(\sum_{n \in S} |a_n|^p \biggr)^{1/p}- \biggl(\sum_{n \in S} |b_n|^p \biggr)^{1/p} \Biggr|.
\end{multline*}
Note that in the previous chain of inequalities, the inequality actually becomes an equality. Indeed, consider some fixed subset $S^* \subset \N$ and fix an arbitrary $\e >0$. Consider a finite set of numbers $0 \le D_1, \dots, D_N < 1$ such that
$\bigl| \bigl(\sum_{n \in S^*} |a_n|^p \bigr)^{1/p} - \|\{a_n D_n\}_{n=1}^{N}\|_{l^p} \bigr| < \e / 2$ and
$\bigl| \bigl(\sum_{n \in S^*} |b_n|^p \bigr)^{1/p} - \|\{b_n D_n\}_{n=1}^{N}\|_{l^p} \bigr| < \e / 2$.
Then
$$\Biggl| \biggl(\sum_{n \in S^*} |a_n|^p \biggr)^{1/p}- \biggl(\sum_{n \in S^*} |b_n|^p \biggr)^{1/p} \Biggr| \le
\Biggl| \biggl(\sum_{n = 1}^{N} |a_n D_n|^p \biggr)^{1/p}- \biggl(\sum_{n = 1}^{N} |b_n D_n|^p \biggr)^{1/p} \Biggr| + \e.$$
Note that we can choose $D_1, \dots D_N$ such that for every $n$, for some
$w_n^*, \tilde{w}_n^* \in W_n$, we have
$D_n = d_{W_n}(w_n^*, \tilde{w}_n^*)$.
Then, by construction, we obtain
\begin{multline*}
\Biggl| \biggl(\sum_{n \in S^*} |a_n|^p \biggr)^{1/p}- \biggl(\sum_{n \in S^*} |b_n|^p \biggr)^{1/p} \Biggr| \le\\
\
\Biggl| \biggl(\sum_{n=1}^{\infty} \bigl(a_n \ d_{W_n}(w_n^*, \tilde{w}_n^*) \bigr)^p \biggl)^{1/p}-
\biggl(\sum_{n=1}^{\infty} \bigl(b_n \ d_{W_n}(w_n^*, \tilde{w}_n^*) \bigr)^p \biggl)^{1/p} \Biggr| + \e,
\end{multline*}
which justifies the desired equality and proves the proposition.
\end{proof}

This proposition implies the following theorem.

\begin{thm}\label{thm:MainTheorem}
\begin{enumerate}
\item \label{thm:MainTheorem:1} Let $1\le p < \infty$. Consider $\{W_n\}_{n \in \N}$ --- a fixed sequence of metric spaces, let $\diam(W_n)=1$ for all $n\in \N$. Let $a,b \in l^p_{+}$.

Let us denote $$X = (l^p)\prod_{n=1}^{\infty}a_n W_n, \ Y=(l^p) \prod_{n=1}^{\infty}b_n W_n.$$

Then $$2d_{GH}(X,Y) \le \sup_{S \subset \N} \ \Biggl| \biggl(\sum_{n \in S} |a_n|^p \biggr)^{1/p}- \biggl(\sum_{n \in S} |b_n|^p \biggr)^{1/p} \Biggr|.$$

\item \label{thm:MainTheorem:2} Consider $\{W_{\a}\}_{\a \in A}$ --- any family
of metric spaces such that $\diam(W_{\a})=1$ for all $\a \in A$.
Let $a = \{a_{\a}\}_{\a \in A}, a_{\a}\ge0$ and $b = \{b_{\a}\}_{\a \in A}, b_{\a}\ge0$ be such that $\sup_{\a \in A}a_{\a} < \infty$, $\sup_{\a \in A}b_{\a} < \infty$.

Let us denote $$X = (l^{\infty})\prod_{\b \in A} a_{\b}W_{\b}, \ Y = (l^{\infty})\prod_{\b \in A} b_{\b}W_{\b}.$$

Then $$2d_{GH}(X,Y) \le \sup_{S \subset A}\big| \sup_{\b \in S}a_{\b} - \sup_{\b \in S}b_{\b}\big|.$$
\end{enumerate}
\end{thm}

\begin{proof}
Indeed, in the notation ~\ref{prop:DisDiagonal}, $2d_{GH}(X, Y) \le \dis(f)$ holds.
\end{proof}

\begin{thm}\label{thm:TheoremEquality}
\begin{enumerate}
\item Let $1 \le p <\infty$. Consider $\{W_n\}_{n \in \N}$ --- a fixed sequence of metric spaces such that $\diam(W_n)=1$ for all $n \in \N$. Let $a,b\in l^{p}_{+}$ be such that
$$\sup_{S \subset \N} \ \Biggl| \biggl(\sum_{n \in S} |a_n|^p \biggr)^{1/p}- \biggl(\sum_{n \in S} |b_n|^p \biggr)^{1/p} \Biggr| \le \bigl|\|a\|_{l^p}-\|b\|_{l^p} \bigr|.$$
Then $$2d_{GH}\Biggl((l^p)\prod_{n=1}^{\infty}a_n W_n, \ (l^p)\prod_{n=1}^{\infty}b_n W_n \Biggr)=\bigl|\|a\|_{l^p}-\|b\|_{l^p} \bigr|.$$
\item Let $\{W_{\a}\}_{\a \in A}$ be an arbitrary family of metric spaces such that $\diam(W_{\a})=1$ for all $\a \in A$. Let $a=\{a_{\a}\}_{\a \in A}, a_{\a}\ge0$, $b=\{b_{\a}\}_{\a \in A},b_{\a}\ge0$ such that $\sup_{\a \in A}a_{\a} < \infty$, $\sup_{\a \in A}b_{\a} < \infty$.
Let's assume that
$$ \sup_{S \subset A}\big| \sup_{\b \in S}a_{\b} - \sup_{\b \in S}b_{\b}\big| \le \big| \sup_{\b \in A}a_{\b} - \sup_{\b \in A}b_{\b}\big|. $$
Then $$2d_{GH}\Biggl((l^{\infty})\prod_{\b \in A}a_{\b} W_{\b}, \ (l^{\infty})\prod_{\b \in A}b_{\b}W_{\b} \Biggr) = \big| \sup_{\b \in A}a_{\b} - \sup_{\b \in A}b_{\b}\big|.$$
\end{enumerate}
\end{thm}

\begin{proof}
Let's prove the first point, the second point is similar. Thus, we denote
$X=(l^p)\prod_{n=1}^{\infty}a_n W_n$, $Y=(l^p)\prod_{n=1}^{\infty}b_n W_n$ and note that $\diam(X)=\|a\|_{l^p}$ and that $\diam(Y)=\|b\|_{l^p}$. By Theorem ~\ref{thm:MainTheorem} and based on the estimate
$$\bigl|\diam(X)-\diam(Y)\bigr|\le 2d_{GH}(X, Y),$$ we obtain $2d_{GH}(X,Y)=\bigl|\|a\|_{l^p}-\|b\|_{l^p} \bigr|$.
\end{proof}

\begin{rk}
Note that Theorems ~\ref{thm:MainTheorem} and ~\ref{thm:TheoremEquality} hold for arbitrary finite products, since the remaining factors can always be formally defined as follows: $X_n=Y_n=\{pt\}$ or $X_{\a}=Y_{\a}=\{pt\}$.
\end{rk}

\begin{cor}
Let $W_n$ be any sequence of metric spaces such that $\diam(W_n)=1$.
Let $a,b\in l^1_{+}$ be such that
for every $n \in \N$, $a_n \le b_n$.
Then
$$2d_{GH}\Biggl((l^{1})\prod_{n=1}^{\infty}a_{n} W_{n}, \ (l^{1})\prod_{n = 1}^{\infty}b_{n}W_{n} \Biggr) = \sum_{n=1}^{\infty}(b_n-a_n).$$
\end{cor}

\begin{proof}
The condition $a_n \le b_n$ automatically guarantees that the conditions of Theorem ~\ref{thm:TheoremEquality} are satisfied.
\end{proof}

\begin{notation}
Denote by $S^1$ the metric space that is a circle with intrinsic metric, where $\diam(S^1)=1$.
\end{notation}

\begin{examp}
Let's estimate the Gromov--Hausdorff distance between the tori $X=(l^2)aS^1 \times bS^1$ and $Y=(l^2)cS^1 \times dS^1$.
By Theorem ~\ref{thm:TheoremEquality} we obtain that if $$\max\{|a-c|, |b-d|\} \le \bigl|\sqrt{a^2+b^2}-\sqrt{c^2+d^2} \bigr|,$$ then $$2d_{GH}(X,Y)=\bigl|\sqrt{a^2+b^2}-\sqrt{c^2+d^2} \bigr|.$$
For example, for $a=1, \ b=3, \ c=2, \ d=5$ this condition is satisfied and, accordingly,
$$2d_{GH}((l^2)S^1 \times 3S^1, \ (l^2)2S^1 \times 5S^1)=\sqrt{29}-\sqrt{10}.$$
\end{examp}

\end{subsection}

\section{Attainability of the Upper Bound for $l^{\infty}$-Products}
In Theorem ~\ref{thm:TheoremEquality}, we obtained sufficient conditions for the attainability of the lower bound in the inequality
$$\bigl|\diam(X)-\diam(Y)\bigr|\le2d_{GH}(X, Y) \le \max\{\diam(X),\diam(Y)\}.$$
Now we obtain a sufficient condition for attaining the upper bound. We begin with the following proposition.

\begin{prop}\label{prop:TheLast}
\begin{enumerate}
\item
Let $X$ and $Y$ be metric spaces such that $\diam(X)\le\diam(Y)<\infty$.
Suppose that for any $\e>0$ there exists a subset $Q_{\e} \subset Y$ such that $\#Q_{\e}>\#X$ and for all $x,y\in Q_{\e}$ we have $d(x,y)\ge \diam(Y)-\e$. Then $2d_{GH}(X,Y)=\diam(Y).$
\item
Let $X$ and $Y$ be metric spaces, and $\diam(Y) = \infty$.
Suppose that for any $C > 0$ there exists a subset $Q_{C} \subset Y$ such that $\#Q_{C}>\#X$ and for all $x,y\in Q_{C}$ we have $d(x,y)\ge C$. Then $d_{GH}(X,Y) = \infty$.
\end{enumerate}
\end{prop}

\begin{proof}
We will only prove the first point, since the second point is similar.
Clearly, $2d_{GH}(X,Y)\le\max\{\diam(X),\diam(Y)\}=\diam(Y)$, so it suffices to prove that $2d_{GH}(X,Y)\ge \diam(Y)$.
So,
\begin{multline*}
2d_{GH}(X,Y)=\inf_{R \in \cR(X, Y)}\dis(R)
=\inf_{f\:X\to Y}\inf_{g\:Y\to X} \max\{\dis(f), \dis(g), \codis(f, g)\} \ge \\
\
\inf_{g\:Y \to X}\dis(g) \ge \inf_{g\:Y \to X} \dis(g|_{Q_{\e}})=\inf_{g\:Q_{\e} \to X} \dis(g) = \\
\
= \inf_{g\: Q_{\e} \to X} \sup_{x, y \in Q_{\e}} \biggl|d_{Y}(x,y)-d_X \bigl(g(x),g(y)\bigr) \biggr| \ge \diam(Y)-\e.
\end{multline*}
The last inequality is explained by the fact that the mapping $g\:Q_{\e}\to X$ cannot be injective, so there exist $x, y\in Q_{\e}$ such that $x \ne y$ and $g(x) = g(y)$. Consequently, for these $x, y$, we have
$$\biggl|d_{Y}(x,y)-d_X \bigl(g(x),g(y)\bigr) \biggr| \ge \diam(Y)-\e.$$
Passing to the limit as $\e \to 0$, we obtain $2d_{GH}(X,Y) \ge \diam(Y)$.
\end{proof}

\begin{cor}
Let $X,Y$ be metric spaces, $\diam(X)\le\diam(Y)$, and $\diam(Y)\D_{\kappa} \subset Y$, where $\#X<\kappa$. Then $2d_{GH}(X,Y)=\diam(Y)$.
\end{cor}
\begin{proof}
For $Q_{\e}$, we can take $\diam(Y)\D_{\kappa}$.
\end{proof}

\begin{notation}[\cite{Engelking}]
For an arbitrary topological space $X$, let $d(X)$ denote its \emph{density} --- the smallest cardinality of its dense subsets.
\end{notation}

\begin{thm}\label{thm:TheLast}
\begin{enumerate}
\item
Let $X$ be a bounded metric space of density $d(X)$. Then, for any set $A$ of cardinality $\#A\ge d(X)$, we have
$$2d_{GH}\Bigl(X, (l^{\infty})\prod_{\a \in A}X \Bigr) = \diam(X).$$
\item
Let $X$ be an unbounded metric space of density $d(X)$. Then, for any set $A$ of cardinality $\#A\ge d(X)$, we have
$$d_{GH}\Bigl(X, (l^{\infty})\prod_{\a \in A}X \Bigr) = \infty.$$
\end{enumerate}
\end{thm}

\begin{proof}
We will prove only the first point, since the second is similar.
Denote by $Y\subset X$ an everywhere dense subset of cardinality $\#Y=d(X)$. Then, it is well known that $\diam(Y)=\diam(X)$ and $d_{GH}(X,Y)=0$.
Thus, for any $\e>0$, consider $x,y\in Y$ such that $d_{X}(x,y)\ge\diam(X)-\e$.
Denoting $$Q_{\e}=(l^\infty)\ \prod_{\a \in A}\{x,y\}\subset (l^{\infty})\prod_{\a \in A}X$$ and taking into account that $\#Q_{\e}=\exp(\#A)>\#A\ge d(X)=\#Y$, by Proposition ~\ref{prop:TheLast} we obtain that
$$2d_{GH}\Bigl(X, (l^{\infty})\prod_{\a \in A}X \Bigr) = 2d_{GH}\Bigl(Y, (l^{\infty})\prod_{\a \in A}X \Bigr) = \diam(X).$$
\end{proof}

\begin{cor}
Let $X$ be a metric space. Then
$$2d_{GH}\Bigl(X, (l^{\infty})\prod_{y \in X}X \Bigr) = \diam(X).$$
\end{cor}

\begin{cor}
Let $X$ be a separable metric space. Then
$$2d_{GH}\Bigl(X, (l^{\infty})\prod_{n=1}^{\infty}X \Bigr) = \diam(X).$$
\end{cor}

\section{Gromov--Hausdorff Topological Distance}

\begin{dfn}
Let $R$ be an equivalence relation on the metric space $(X, d)$.
\emph{The semimetric $d_R$ of the quotient space} is defined by the formula
$$d_R(x, y) = \inf \Bigl\{ \sum_{i = 1}^{k} d(p_i, q_i) \: p_0 = x, q_k = y, k \in \N \Bigr\},$$
where the infimum is taken over all sets $\{p_i\}$ and $\{q_i\}$ such that the point $q_i$ is $R$-equivalent to the point $p_{i + 1}$ for all $i = 1, \dots, k - 1$.
\end{dfn}

The question arises: can the Gromov–Hausdorff distance between arbitrary metrizable spaces be defined as a topological characteristic?
If we consider the induced semimetric on the space of homeomorphism classes $\cGH$, it is easy to verify that this semimetric is trivial (identically zero).
Indeed, for any metrizable space $X$, we can always introduce a metric $d_{\e}$ that generates this topology, in which $\diam(X)<\e$ for any given $\e > 0$. Since the distance from the space $\{pt\}$ to $X$ will not exceed $\frac{1}{2} \diam(X)$, by virtue of the triangle inequality, we obtain that this distance is trivial.

It is natural to assume that the triviality is due to the possibility of an infinitely decreasing diameter. We will attempt to eliminate this effect by restricting the factorization to subsets of compact spaces of fixed diameter (e.g., $\diam(X) = 1$). It seems likely that in this class, the induced semimetric will allow us to distinguish non-homeomorphic spaces. However, as will be shown below, even this restriction does not prevent distance degeneration.

Recall that $\cM$ denotes the space of all metric compact spaces.
Denote by $\cS$ the sphere
$$\cS = \bigl\{ X \in \cM \: \diam(X) = 1 \bigr\}.$$
We introduce the following equivalence relation $R$: metric spaces $X, Y \in \cS$ are $R$-equivalent if they are homeomorphic, and we denote the resulting semimetric, the metric on the quotient space, by $D$.

\begin{lem}\label{lem:lem}
Let $X, Y \in \cS$.
Denote $X \times Y = (l^{\infty}) X \times Y$.
Then $D(X, X \times Y)$ = 0.
\end{lem}

\begin{proof}
Fix $0 < \lambda < 1$.
Note that since the spaces $X \times Y$ and $X \times \lambda Y$ are homeomorphic, then by Theorem \ref{thm:EstimationInGeneralCase_L_infty} we have $D(X, X \times Y) = D(X \times \{pt\}, X \times \lambda Y) \le d_{GH}(X \times \{pt\}, X \times \lambda Y) \le
\max\{d_{GH}(X, X), d_{GH}(\{pt\}, \lambda Y)\} = \diam(\lambda Y) / 2 = \lambda/2$.
Letting $\lambda \to 0$, we obtain the desired result.
\end{proof}

\begin{thm}
The semimetric $D$ of the quotient space $\cS$ under the equivalence relation $R$ is identically zero.
\end{thm}

\begin{proof}
In the notation of Lemma \ref{lem:lem}, for any metric spaces $X, Y \in \cS$, we have
$$D(X, Y) \le D(X, X \times Y) + D(X \times Y, Y) = 0.$$
\end{proof}

\end{document}